\documentclass[12pt]{amsart}
\usepackage{mathrsfs}
\usepackage{amsfonts}
\usepackage{amssymb, amstext, amscd, amsmath}
\usepackage{amsthm}
\usepackage{graphicx}
\usepackage{txfonts}
\usepackage{enumerate}

\newtheorem{thm}{Theorem}[section]
\newtheorem{cor}[thm]{Corollary}

\newtheorem{lem}[thm]{Lemma}

\theoremstyle{definition}

\numberwithin{equation}{section}

\begin{document}

\title[ On Heegaard splittings with finite many pairs of disjoint compression disks]
{ On Heegaard splittings with finite many pairs of disjoint compression disks}

\author{Qiang E}
\address{School of science, Dalian Maritime University, Dalian, P.R.China} \email{eqiang@dlmu.edu.cn}
\author{Zhiyan Zhang}
\address{School of science, Dalian Maritime University, Dalian, P.R.China} \email{18234504990@qq.com}

\thanks{This research is supported by  grants
 of NSFC (No.11401069 and No.11671064), }



\begin{abstract}
Suppose $V\cup_S W$ is a weakly reducible Heegaard splitting of a closed 3-manifold which admits only $n$ pairs of disjoint compression disks on distinct sides and $g>2$. We show $V\cup_S W$ admits an untelescoping:
$(V_1\cup_{S_1}W_1)\cup_F(W_2\cup_{S_2}V_2)$ such that $W_i$ has only one separating compressing disk and $d(S_i)\geq 2$, for $i=1,~2$. If $n>1$, at least one of $d(S_i)$ is 2 and $S$ is a critical Heegaard surface.
\end{abstract}
\maketitle

\section{Introduction}
Let $M$ be a connected, orientable, closed 3-manifold. Let $V\cup_S W$ be a Heegaard splitting of $M$, thus $V$ and $W$ are handlebodies which share the common boundary suface $S$. $V\cup_S W$ is said to be weakly reducible if there are two essential disks $D\subset V$ and $E\subset W$, such that $ D\cap E=\emptyset$.  A Heegaard splitting which is not weakly reducible is said to be strongly irreducible. Casson and Gorden proved that if a 3-manifold admits an irreducible but weakly reducible Heegaard splitting,  it must contain an incompressible surface\cite{CG}. Schlarmann and Thompson prove that if a Heegaard splitting is weakly reducible, then it admits an untelescoping by reattaching handles\cite{ST}.

In this paper, the authors are interested in a weakly reducible Heegaard splitting which admits finite many pairs of disjoint compression disks on distinct sides up to isotopy. The first author proved that for each $g>2$, there are infinitely many genus g Heegaard splittings which admits a unique pair of disjoint compression disks on distinct sides\cite{EQ}. Such a Heegaard splitting is said to be keen weakly reducible. Keen weakly reducible Heegaard surfaces are interesting because they are examples of topologically non-minimal surfaces which are introduced by David Bachman\cite{GT}.

In Section 3, we will show if $V\cup_S W$ admits an untelescoping $(V_1\cup_{S_1}W_1)\cup_F(W_2\cup_{S_2}V_2)$, such that $W_i$ has only one separating compression disk and $d(S_i)\geq 3$, for $i=1,~2$, then the amalgamated Heegaard surface $S$ is keen weakly reducible. This improves our previous result in \cite{EQ}, which is under the similar assumption but $d(S_i)\geq 4$ is required.

Although we could not answer whether there exists a weakly reducible Heegaard splitting which admits more than one but finite many pairs of disjoint compression disks on distinct sides, we will give some necessary conditions to describe properties of such Heegaard splittings in Section 4.

\section{Preliminaries}

Let $S$ be a closed orientable surface whose genus is at least 2. The distance between two essential simple closed curves
$\alpha$ and $\beta$ in $S$, denoted by $d_S(\alpha, \beta)$, is the
smallest integer $n \geq 0$ such that there is a sequence of
essential simple closed curves $\alpha=\alpha_{0},
\alpha_{1},...,\alpha_{n}=\beta$ in $S$ where $\alpha_{i-1}$ is
disjoint from $\alpha_{i}$ for $1\leq i\leq n$. If $S$ is an embedded surface in a 3-manifold, and $D$ and $E$ are two compression disks on distinct sides, sometimes $d_S(\partial D,\partial E)$ is denoted simply by $d_S(D, E)$. Let $A$ and $B$ be two sets of essential simple closed curves in $S$. The distance between $A$ and $B$, which is denoted by $d_S(A,B)$, is defined to be
$min\{d_S(x,y)|x\in A, y\in B\}$.  The Heegaard distance of a Heegaard splitting $V\cup_{S}W$ is defined to be $d(S)=d_S(\mathscr{D}_V,\mathscr{D}_W)$ where $\mathscr{D}_V$ and $\mathscr{D}_W$ are sets of essential disks in $V$ and $W$, respectively. $V\cup_{S}W$ is weakly reducible if and only if $d(S)\leq 1$.
The Heegaard distance was first defined by Hempel, see \cite{JH}.

\begin{lem}\label{5}
Let $S$ be a closed orientable surface. Suppose that $\beta$ and $\gamma$ are two essential simple closed curves on $S$. If $d_S(\beta, \gamma)\geq 3$ and $\beta$ is separating on $S$, then there are at least two essential sub-arcs of $\gamma$ on each component of $S\setminus\beta$.
\end{lem}
\begin{proof}
We assume that $|\beta\cap\gamma|$ is minimal in the isotopy classes of $\beta$ and $\gamma$. Since $\beta$ is separating on $S$, $|\beta\cap\gamma|$ is even. If $|\beta\cap\gamma|=0$ then $d_S(\beta, \gamma)\leq 1$. If $|\beta\cap\gamma|=2$ then one boundary component of the closure of $N(\beta\cup\gamma)$ is essential on $S$ and disjoint from $\beta$ and $\gamma$ and we have $d_S(\beta, \gamma)=2$. Therefore  $|\beta\cap\gamma|\geq 4$ and there are at least two essential sub-arcs of $\gamma$ on each component of $S\setminus\beta$.
\end{proof}

Let $V_1$ be a handlebody whose genus is at least 2 and $\partial V_1=S_1$. Let $V$ be the handlebody obtained by attaching one 1-handle $D_0\times I$ to $V_1$ along a pair of disjoint disks $D_1, D_2\subset S_1$, where $D_0$ is the disk corresponding to the 1-handle. $\mathscr{D}_V$ is defined to be the set of compression disks of $V$. There is a partition of $\mathscr{D}_V$: $\mathscr{D}_V=\mathscr{D}_0\cup\mathscr{D}_1\cup\mathscr{D}_2\cup\mathscr{D}_3$, such that, $\mathscr{D}_0=\{D_0\}$, $\mathscr{D}_1=\{D~| D\cap D_0=\emptyset;~ D\neq D_0, D$ is inessential in $V_1\}$, $\mathscr{D}_2=\{D~| D\cap D_0=\emptyset;~ D$ is essential in $V_1\}$ and $\mathscr{D}_3=\{D~| D\cap D_0\neq\emptyset\}$.  It is clear that any compression disk $D$ of $V$ belongs to one and only one of the four subsets up to isotopy.  If $D\in \mathscr{D}_1$, then  $\partial D,~\partial D_1$ and $\partial D_2$ co-bound a pair of pants on $S_1$. Hence $D$ is isotopic to a band-sum of $D_1, D_2$ along an arc and $\partial D$ bounds a once-punctured torus. $D\in \mathscr{D}_2$ if and only if $D\in \mathscr{D}_{V_1}$.

Now we consider $D\in \mathscr{D}_3$, that is, $D$ is an essential disk in $V$ such that $D\cap D_0\neq\emptyset$.  Furthermore, $D$ is isotoped in $V$ such that $|D\cap D_0|$ is minimal. Let $S_1^*$ be the surface  $S_1-(int(D_1)\cup int(D_2))$. Then $S_1^*$ is a sub-surface of $S_1$ with
two boundary components $\partial D_1$ and $\partial D_2$. By standard arguments, we have some observations as follows.

\begin{lem}\label{1}\cite{ZY}
\mbox\par
\begin{enumerate} [(1)]
\item   Each component of $D\cap D_0$ is a properly embedded arc in both $D$ and $D_0$.
\item   Each component of $\partial D\cap S_1^*$ is essential on $S_1^*$.
\item   Each component of $D\cap(\partial D_0\times  I)$ is an arc with its two end points lying in distinct boundary components of the annulus $\partial D_0\times  I$.
\end{enumerate}
\end{lem}

 Let $\gamma$ be an outermost component of $D\cap(D_1\cup D_2)$ on $D$. This means that $\gamma$, together with an arc $\gamma_1\subset\partial D$, bounds a sub-disk in $D$, say $D_\gamma$, such that $D_\gamma\cap (D_1\cup D_2)=\gamma$. We call $\gamma_1$ \emph{an outermost arc related to $\gamma$} and call $D_\gamma$ \emph{an outermost disk related to $\gamma$}. See Figure \ref{d3}.

\begin{lem}\label{2}\cite{ZY}

\mbox\par

\begin{enumerate} [(1)]
\item $\gamma_1$,  whose end points lie in one of $D_1$ and $D_2$, is strongly essential in $S_1^*$.
\item $D_\gamma\in\mathscr{D}_2$, that is, $D_\gamma$ is an essential disk both in $V_1$ and $V$.
\end{enumerate}

\end{lem}

We explain that an essential arc in $S_1^*$ is called \emph{strongly essential} if both boundary points lie in $\partial D_i$ and it is an essential arc on $S_1\cup D_j$ , where $\{i, j\} =\{1, 2\}$.
\begin{figure}
\centering
    \includegraphics[width=5cm]{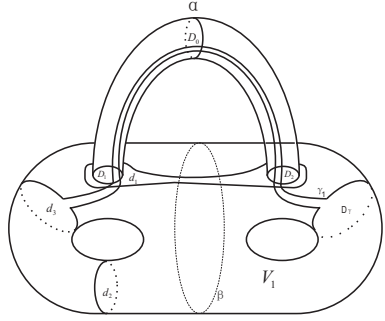}\\
  \caption{disks in the handlebody such that $d_i\in \mathscr{D}_i$}\label{d3}
\end{figure}

\section{keen weakly reducible Heegaard splittings}
\begin{thm}\label{6}
Let $V\cup_S W$ be a Heegaard splitting of a closed 3-manifold $M$. Suppose it admits an untelescoping:
$$M=V\cup_S W=(V_1\cup_{S_1}W_1)\cup_F(W_2\cup_{S_2}V_2),$$  such that $W_i$ has only one separating compressing disk and $d(S_i)\geq 3$, for $i=1,~2$,
then  $S$ is keen weakly reducible.
\end{thm}
\begin{proof}  The untelescoping of $V\cup_S W$ could be realized by reattaching handles as follows:

$V_i$ is a genus $g>1$ handlebody and  $S_i=\partial V_i$, for $i=1,~2$. There exists an essential separating curve  $\beta\subset S_1$, such that $d_{S_1}(\beta, \mathscr{D}_{V_1})\geq 3$ and we attach a 2-handle $E_0\times I$ to $V_1$ along $\beta\times I$ so that $\beta$ bounds the unique disk $E_0\subset W_1$. Let $M_1 =V_1\cup_{\beta\times I}(E_0\times I)=V_1\cup_{S_1}W_1$ and $F=\partial M_1$ which has two components.

Next, we attach a 1-handle $D_0\times I$ to $M_1$ such that $(D_0\times I)\cap(E_0\times I)=\emptyset$ and the gluing disks denoted by $D_1\cup D_2=V_1\cap (D_0\times I)$ are on distinct components of $F$. $D_0$, corresponding to the 1-handle, is the unique disk in $W_2$ and we denote $\partial D_0$ by $\alpha$. Let $M_2=M_1\cup_{D_1\cup D_2}(D_0\times I)=(V_1\cup_{S_1}W_1)\cup_F W_2$ be the resulted 3-manifold.

At last, we attach $V_2$ to $M_2$ via some orientation preserving homeomorphism $f:\partial M_2\rightarrow S_2$ such that
$d_{S_2}(f(\alpha), \mathscr{D}_{V_2})\geq 3$. Since $\alpha$ is separating on $\partial M_2$, $f(\alpha)$ is separating on $S_2$ and after the handlebody-attaching, we will use $\alpha$ instead of $f(\alpha)$. $V_2\cap (E_0\times I)$ are two disks, denoted by $E_1$,$E_2$. The result is the closed 3-manifold $M$, that is
 $$M=(V_1\cup_{S_1}W_1)\cup_F(W_2\cup_{S_2}V_2)=V_1\cup_{\beta\times I}(E_0\times I)\cup_{D_1\cup D_2}(D_0\times I)\cup_{S_2}V_2.$$

Now let $V=V_1\cup_{D_1\cup D_2} (D_0\times I)$, $W=V_2\cup_{E_1\cup E_2} (E_0\times I)$ and $S=V\cap W$. Then both $V$ and $W$ are genus $g+1$ handlebodies and $M=V\cup_SW$ is a Heegaard splitting.  $\alpha$ bounds the disk $D_0$ in $V$ and $\beta$ bounds the disk $E_0$ in $W$. $D_0\cap E_0=\emptyset$ implies that $M=V\cup_SW$ is weakly reducible.

In order to prove $(D_0, E_0)$ is the unique pair of disjoint compression disks of $S$, we divide $\mathscr{D}_V$ and $\mathscr{D}_W$ as mentioned in Section 2: $\mathscr{D}_V=\mathscr{D}_0\cup\mathscr{D}_1\cup\mathscr{D}_2\cup\mathscr{D}_3,$ such that $\mathscr{D}_0=\{D_0\}$. $\mathscr{D}_1=\{D~| D\cap D_0=\emptyset;~ D\neq D_0; D$ is inessential in $V_1\}$. $\mathscr{D}_2=\{D~| D\cap D_0=\emptyset;~ D$ is essential in $V_1\}= \mathscr{D}_{V_1}$. $\mathscr{D}_3=\{D~| D\cap D_0\neq\emptyset\}$. Similarily, $\mathscr{D}_W=\mathscr{D}^0\cup\mathscr{D}^1\cup\mathscr{D}^2\cup\mathscr{D}^3$, such that  $\mathscr{D}^0=\{E_0\}$. $\mathscr{D}^1=\{E~| E\cap E_0=\emptyset;~ E\neq E_0; E$ is inessential in $V_2\}$. $\mathscr{D}^2=\{E~| E\cap E_0=\emptyset;~ E$ is essential in $V_2\}= \mathscr{D}_{V_2}$. $\mathscr{D}^3=\{E~| E\cap E_0\neq\emptyset\}$. We will show that for each $D\in\mathscr{D}_i$, $E\in\mathscr{D}^j$, where $i,j=0,1,2,3$ and $(i,j)\neq(0,0)$,  $D\cap E\neq\emptyset$ holds. Notice that if it holds for some $(i,j)$, then by the symmetric construction of the Heegaard splitting and the similar partitions of the $\mathscr{D}_V$ and $\mathscr{D}_W$, it also holds for $(j,i)$. Thus we just need to prove under the case $i\leq j$.

In the following argument, we suppose to the contrary that $D\cap E=\emptyset$ and assume that $|D\cap D_0|+|E\cap E_0|$ is minimal in the isotopy classes of $D$ and $E$.

\newtheorem{clm}{Case}

\begin{clm}\label{01}
$D\in\mathscr{D}_0$.
\end{clm}

  In this case $D=D_0$ .
  If $E\in\mathscr{D}^1$, $E$ is a band-sum of $E_1$ and $E_2$ along an arc on $S_2$. Since $\alpha$ separates $E_1$ and $E_2$ in $S_2$, any arc connected them must intersect $\alpha=\partial D_0$. It follows that $D\cap E\neq\emptyset$. If $E\in\mathscr{D}^2$, $E$ is an essential disk in $V_2$. Since
$d_{S_2}(D_0,\mathscr{D}_{V_2} )=d_{S_2}(f(\alpha),\mathscr{D}_{V_2} )\geq 3>2$, $D_0$ intersects each compression disk of $V_2$. It follows that $D\cap E\neq\emptyset$. If $E\in\mathscr{D}^3$, $E\cap E_0\neq\emptyset$. By applying Lemma \ref{2} to $W$ and $V_2$, there is an outermost disk of $E$, say $E_\gamma$, such that $E_\gamma\in \mathscr{D}^2$. By the above discussion, $D\cap E=D_0\cap E\supset D_0\cap E_\gamma\neq\emptyset$. Therefor when
$d(S_i)\geq 2$, any disk in $\mathscr{D}_W\backslash E_0$ intersects $D_0$, and any disk in $\mathscr{D}_V\backslash D_0$ intersects $E_0$.

\begin{clm}\label{11}
$D\in\mathscr{D}_1$.
\end{clm}

In this case, $D$ is a band-sum of $D_1$ and $D_2$ along an arc. We denote the arc by $c$.  $\beta$ separates $D_1$ and $D_2$ in $S_1$ implies that $\beta$ intersects $c$ as well as the band $N(c)$. Notice that $\partial D$ bounds a once-punctured torus $T_D=(\alpha\times I)\cup N(c)$ on $S$. Since $\alpha\cap\beta=\emptyset$, $\beta\cap T_D=\beta\cap N(c)$. If $\beta$ intersects $c$ only once, then $\overline{T_D\setminus\beta}$ is isotopic to $\alpha\times I$; Otherwise $\overline{T_D\setminus\beta}$ contains a component isotopic to $\alpha\times I$ and some disks.

Since $E\cap D=\emptyset$, $\partial E\subset T_D$ or $\partial E\subset (S\setminus T_D)$. If $\partial E\subset (S\setminus T_D)$ then $E\cap D_0=\emptyset$ which contradicts Case \ref{01}. Thus $\partial E\subset T_D$.

If $E\in\mathscr{D}^1$, $\partial E$ also bounds a once punctured torus $T_E$. $\partial E\subset T_D$ implies that $\partial E$ is isotopic to $\partial D$. It follows that $\alpha$ and $\beta$ are isotopic because $\alpha\subset T_D$, $ \beta\subset T_E$ and $\alpha\cap\beta=\emptyset$. But $\alpha$ and $\beta$ are not isotopic,  a contradiction.

If $E\in \mathscr{D}^2$, $E\cap E_0=\partial E\cap \beta=\emptyset$ and $\partial E$ is essential in $S$. It follows that $\partial E\subset \overline{T_D\setminus\beta}$ and $\partial E$ cannot lie in any disk component of $\overline{T_D\setminus\beta}$.  It follows that after isotopy, $\partial E\subset\alpha\times I \subset S$. In this case $\partial E$ is isotopic to $\alpha$ which implies that $d_{S_2}(f(\alpha),\mathscr{D}_{V_2} )=0$, a contradiction.

If $E\in\mathscr{D}^3$, $E\cap E_0\neq\emptyset$. By applying Lemma \ref{2} to $W$ and $V_{2}$, there is an outermost arc, say $\gamma$, whose two endpoints lie in $\partial E_i$, where $i=1$ or $2$. Furthermore, there is an outermost disk of $E$, say $E_{\gamma}$, such that $E_{\gamma}$ is essential in $V_{2}$. $E_{\gamma}\cap E_i$ is an arc denoted by $e_i$.
On the other hand, $D$ is a band-sum of $D_1$ and $D_2$ which lie distinct sides of $\beta$. Hence we may find a sub-arc $\gamma_1$ of $\partial D$ such that $\partial\gamma_1\subset\partial E_j$, where $j=3-i$.  $\gamma_1$, together with a sub-arc of $\partial E_j$, say $e_j$, bounds a disk isotopic to $D_0$. $D\cap E=\emptyset$ implies that $\gamma\cap\gamma_1=\emptyset$. Since $e_i\subset E_i$, $e_j\subset E_j$ and $E_i\cap E_j=\emptyset$, we have $e_i\cap e_j=\emptyset$.  Thus $(\gamma\cup e_i)\cap(\gamma_1\cup e_j)=\emptyset$ and it means that $E_{\gamma}\cap D_0=\emptyset$. This contradicts Case \ref{01}.

\begin{clm}\label{22}
$D\in\mathscr{D}_2$.
\end{clm}

In this case, $D$ is essential in $V_1$. If $E\in\mathscr{D}^2$, $E$ is essential in $V_2$.
Since $\partial D\cap\beta\neq\emptyset$ and $\beta$ is separating in $S_1$, there is a sub-arc $\gamma_1$ of $\partial D$ such that $\partial\gamma_1\subset\partial E_1$ and $\gamma_1$ together with an arc of $E_1$ forms an essential closed curve $\gamma$ on $S_2$.  Moreover, $D_0\cap D=\emptyset$ and $D_0\cap E_1=\emptyset$ mean that $D_0\cap\gamma=\emptyset$. $E\cap D=\emptyset$ and $E\cap E_1=\emptyset$ mean that $E\cap\gamma=\emptyset$.
Hence
 $d_{S_2}(f(\alpha),\mathscr{D}_{V_2})\leq d_{S_2}(D_0,E)\leq d_{S_2}(D_0,\gamma)+d_{S_2}(\gamma, E)=1+1=2$, a contradiction.

If $E\in\mathscr{D}^3$, $E\cap E_0\neq\emptyset$. By applying Lemma \ref{2} to $W$ and $V_2$, there is an outermost arc, say $\gamma_1$, whose two endpoints lie in $\partial E_i$, where $i=1$ or $2$. Furthermore, there is an outermost disk of $E$, say $E_{\gamma_1}$, such that $E_{\gamma_1}$ is essential in $V_2$. $E_{\gamma_1}\cap E_i$ is an arc denoted by $e_i$.
On the other hand, since $\partial D$ is essential in $S_1$, $d_{S_1}(\partial D,\beta)\geq 3$ and $\beta$ separates $S_1$, by Lemma \ref{5} there is a sub-arc $\gamma_2$ of $\partial D$ such that $\partial\gamma_2\subset\partial E_j$ where $j=3-i$ and $\gamma_2$ together with an arc of $e_j\subset E_j$ forms an essential closed curve $\gamma$ on $S_2$ such that $D_0\cap\gamma=\emptyset$. See Figure \ref{yb}. $D\cap E=\emptyset$ implies that $\gamma_1\cap\gamma_2=\emptyset$. Since $e_i\subset E_i$, $e_j\subset E_j$ and $E_i\cap E_j=\emptyset$, we have $e_i\cap e_j=\emptyset$.
Hence $E_{\gamma_1}\cap\gamma=\partial E_{\gamma_1}\cap\gamma=(\gamma_1\cup e_i)\cap(\gamma_2\cup e_j)=\emptyset$. Thus
$d_{S_2}(f(\alpha),\mathscr{D}_{V_2})\leq d_{S_2}(D_0,E_{\gamma_1}) \leq d_{S_2}(D_0,\gamma)+d_{S_2}(\gamma, E_{\gamma_1})=1+1=2$, a contradiction.

\begin{clm}
$D\in\mathscr{D}_3$.
\end{clm}

In this case, $D\cap D_0\neq\emptyset$. We only need to discuss when $E\in\mathscr{D}^3$, that is, $E\cap E_0\neq\emptyset$.  By Lemma \ref{2}, there exists an outermost arc $\gamma_1$ of $\partial D$ and an outermost disk of $D$ , say $D_{\gamma_1}$ which is essential in $V_{1}$.
Since $d_{S_1}(\beta,D_{\gamma_1})\geq d_{S_1}(\beta,\mathscr{D}_{V_1})\geq 3$ and $\beta$ separates $S_1$, by Lemma \ref{5}, there exist two sub-arcs of $\gamma_1$, say $\gamma_{11}$ and $\gamma_{12}$, such that
$\partial\gamma_{11}\subset\partial E_1$ and $\partial\gamma_{12}\subset\partial E_2$.  For each $i=1,2$, $\gamma_{1i}$ together with an arc $e_i\subset E_i$ forms an closed curve, say $\gamma_{e_i}$, which is essential in $S_2$.
Since the outermost arc of $D$ and $E_j$ are disjoint from $\alpha$, $D_0\cap\gamma_{e_{i}}=\emptyset$.

 By applying Lemma \ref{2} to $W$ and $V_2$, there is an outermost arc of $\partial E$, say $\gamma_2$, where $\partial \gamma_2\subset \partial E_j$, $j=1$ or $2$.  $\gamma_2$, together with an arc in $E_j$, bounds an outermost disk, say $E_{\gamma_2}$, which is essential in $S_2$. $D\cap E=\emptyset$ means that $\gamma_{11}\cup\gamma_{12}$ and $\gamma_2$ are disjoint. For each case  $\partial \gamma_2\subset \partial E_j$, $j=1,2$, $\gamma_{e_{3-j}}$ and $E_{\gamma_2}$ are disjoint. It follows that $d_{S_2}(f(\alpha),\mathscr{D}_{V_2})\leq d_{S_2}(D_0,E_{\gamma_2}) \leq d_{S_2}(D_0,\gamma_{e_{3-j}})+d_{S_2}(\gamma_{e_{3-j}}, E_{\gamma_2})\leq 2$ which is a contradiction.

Hence for each $D\in\mathscr{D}_i$, $E\in\mathscr{D}^j$, where $(i,j)\neq(0,0)$,  $D\cap E=\partial D\cap \partial E\neq\emptyset$ holds. It means that ($D_0$, $E_0$) is the unique disjoint pair of compression disks on distinct sides of $S$. This completes the proof.
\end{proof}

\section{ Weakly reducible Heegaard splittings with finite many pairs of disjoint compression disks}

In this section, we discuss a weakly reducible Heegaard splitting with more than one but finite many pairs of disjoint compression disks. It is uncertain whether there exists such a weakly reducible Heegaard splitting  but we will give some necessary conditions to describe some properties of such a Heegaard splitting if it does exist.

\begin{lem}\label{3}
Let $V\cup_S W$ be a Heegaard splitting of a closed 3-manifold $M$, where $g>2$. Suppose there are only $n$ pairs of disjoint compression disks
$$\{(D_i,E_i)|D_i\subset V, E_i\subset W,D_i\cap E_i=\emptyset, i=1,2,...,n\},$$
we have
\begin{enumerate} [(1)]
\item for each $i$, $\partial D_i$ and  $\partial E_i$ are not isotopic, moreover,  each is non-separating on $S$, and $\partial D_i\cup\partial E_i$  is separating on $S$.
\item If $D_i=D_j$ for some $i,j$, then $E_i\cap E_j\neq\emptyset$. Similarly, If $E_i=E_j$ for some $i,j$, then $D_i\cap D_j\neq\emptyset$;

\end{enumerate}

\end{lem}

\begin{proof} If for some $i$, $\partial D_i$ and  $\partial E_i$ and $\partial D_i\cup\partial E_i$  are all non-separating on $S$ then any band-sum of two copies of $D_i$ along an arc which is disjoint from $\partial E_i$ is disjoint from $E_i$. There are infinitely many such arcs hence there are infinity many pairs of disjoint disks, a contradiction.

If for some $i$, $\partial D_i$ is separating on $S$, then we may find a non-separating disk $D'_i$ such that $D'_i\cap E_i=\emptyset$. If $\partial E_j$ is also separating on $S$, then we may a non-separating disk $E'_i$ such that $D'_i\cap E'_i=\emptyset$. Since $D'_i\cup E_i$(or $D'_i\cup E'_i$) could not be separating on $S$, by the above discussion , there are infinity many pairs of disjoint disks, a contradiction.

If for some $i$,  $\partial D_i$ and  $\partial E_i$ are isotopic, $V\cup_S W$ is reducible and we may also find two separating disk $\partial D'_i$ and  $\partial E'_i$ which are disjoint, by the above discussion , there are infinity many pairs of disjoint disks, a contradiction.

If for some $i,j$, $D_i=D_j$ and $E_i\cap E_j=\emptyset$, then $E_i\cup E_j$ bounds a sub-surface $S_1$ of $S$. $S_1$ is not an annulus and $D_i\cap S_1=\emptyset$.
Any band-sum of $E_i\cup E_j$ along an arc on $S_1$ is disjoint from $D_i$. There are infinitely many such arcs hence there are infinity many pairs of disjoint disks, a contradiction.

\end{proof}

\begin{thm}
Suppose $V\cup_S W$ is a weakly reducible Heegaard splitting of a closed 3-manifold $M$ which admits  only $n(1\leq n<+\infty)$ pairs of disjoint compression disks on distinct sides and $g>2$. Then it admits an untelescoping:
$$M=V\cup_S W=(V_1\cup_{S_1}W_1)\cup_F(W_2\cup_{S_2}V_2),$$ such that $W_i$ has only one separating compressing disk and $d(S_i)\geq 2$, for $i=1,~2$. If $n>1$, at least one of $d(S_i)$ is 2.
\end{thm}
\begin{proof}
For any pair of disjoint disk $D_i$ and $E_i$, $\partial D_i$ and  $\partial E_i$ are not isotopic, moreover,  each is non-separating on $S$, and $\partial D_i\cup\partial E_i$  is separating on $S$. $E_i$ is corresponding to the 2-handle attaching to $V_1=\overline{V\setminus N(D_i)}$ and $D_i$ is corresponding to the 2-handle attaching to $\overline{W\setminus N(E_i)}$. If $d(S_1)\leq 1$, there exist a disk $D$ in $V_1$ such that $D\cap E_i=\emptyset$. Since $D_i\cap E_i=\emptyset$, by Lemma\ref{3} we have $D\cap D_i\neq\emptyset$. But $D\subset V_1$ and $D_i$ is corresponding to the 1-handle attached to $V_1$, so that $D\cap D_i=\emptyset$, a contradiction. If $n>1$, by Theorem\ref{6}, at least one of $d(S_i)$ is 2.
\end{proof}

An embedded surface $S$ is said to be critical if it is an index 2 tologically minimal surface. There is an equivalent definition of critical surface\cite{GT}: the compression disks for $S$ can be partitioned into two sets $C_0$ and $C_1$, such that there exists at least one pair of disks $D_i, E_i\in C_i$ on opposite sides of $S$, such that $D_i\cap E_i =\emptyset$, for $i=0, 1$. On the other hand, if $D\in C_i$ and $E\in C_{1-i}$ lie on opposite sides of $S$, then
$D\cap E\neq\emptyset$.

\begin{cor}
Suppose $V\cup_S W$ is a weakly reducible Heegaard splitting of a closed 3-manifold $M$ with more than one but finite many of disjoint compression disks on distinct sides and $g(S)>2$. Then $S$ is a critical Heegaard surface.
\end{cor}

\begin{proof}
It follows that $M=V\cup_S W=(V_1\cup_{S_1}W_1)\cup_F(W_2\cup_{S_2}V_2),$ such that $W_i$ has only one separating compressing disk and $d(S_i)\geq 2$, for $i=1,~2$ while at least one of $d(S_i)$ is 2. Let $C_0$ be $\{D_0,E_0\}$ and $C_1$ be other essential disks of $S$.
The Case 1 in the proof of Theorem \ref{6} shows that any disk in $\mathscr{D}_W\backslash E_0$ intersects $D_0$, and
any disk in $\mathscr{D}_V\backslash D_0$ intersects $E_0$.  Since ($D_0$, $E_0$) is not the unique disjoint pair of compression disks, there exist $D\in\mathscr{D}_V\backslash D_0$ and $E\in\mathscr{D}_W\backslash E_0$, such that $D\cap E=\emptyset$. Therefore $C_0$ and $C_1$ satisfies the definition of criticality and $S$ is a critical Heegaard surface.
\end{proof}

\bibliographystyle{amsplain}
\bibliographystyle{amsplain}

\begin{thebibliography}{10}
\bibitem{CG} A.J.Casson, C.McA.Gordon, Reducing Heegaard splittings,Topology Appl. 27(1987) 275-283.
\bibitem{ST} M. Scharlemann and A. Thompson, Thin position for 3-manifolds, AMS Contemporary Math. 164 (1994) 231-238.
\bibitem{EQ}  Q. E, On keen weakly reducible Heegaard splittings. Topology Appl. 231 (2017):128-135.
\bibitem{GT} D. Bachman,  Topological index theory for surfaces in 3-manifolds, Geom. Topol. 14 (2010), 585-609.
\bibitem{JH} J.Hempel, 3-manifolds as viewed from the curve complex, Topology 40(2001) 631-657.
\bibitem{ZY}  Y. Zou, K. Du, Q.Guo and R.Qiu, Unstabilized self-amalgamation of a Heegaard splitting. Topology Appl. 160 (2013), no. 2, 406-411.
\end{thebibliography}

\end{document}